\PassOptionsToPackage{colorlinks, pdftex,citecolor=blue}{hyperref}
\documentclass{amsart}
\usepackage[nobreak]{cite}
\usepackage[utf8]{inputenc}
\usepackage[T1]{fontenc}
\usepackage{hyperref,amssymb,mdwlist,microtype,mathtools,enumitem,color,soul}
\usepackage[capitalize,nameinlink]{cleveref}
\allowdisplaybreaks

\newtheorem{theorem}{Theorem}[section]

\newtheorem{lemma}[theorem]{Lemma}
\newtheorem{prop}[theorem]{Proposition}
\theoremstyle{definition}

\newtheorem{problem}[theorem]{Problem}
\theoremstyle{remark}
\newtheorem{remark}[theorem]{Remark}

\crefname{equation}{}{}
\Crefname{enumi}{}{}

\DeclareMathOperator{\id}{id}

\newcommand{\ignoreSpellCheck}[1]{#1}

\begin{document}

\title[\ignoreSpellCheck{MW}--problem]
{Another look at the Matkowski and Wesołowski problem yielding a new class of solutions}
\author[J.~Morawiec]{Janusz Morawiec}
\address{Institute of Mathematics, University of Silesia, Bankowa 14, PL-40-007 Katowice, Poland}
\email{morawiec@math.us.edu.pl}

\author[T.~Zürcher]{Thomas Zürcher}
\address{Institute of Mathematics, University of Silesia, Bankowa 14, PL-40-007 Katowice, Poland}
\email{thomas.zurcher@us.edu.pl}


\begin{abstract}
The following MW--problem was posed independently by Janusz Matkowski and Jacek Wesołowski in different forms in 1985 and 2009, respectively: Are there increasing and continuous functions $\varphi\colon [0,1]\to [0,1]$, distinct from the identity on $[0,1]$, such that $\varphi(0)=0$, $\varphi(1)=1$ and $\varphi(x)=\varphi(\frac{x}{2})+\varphi(\frac{x+1}{2})-\varphi(\frac{1}{2})$ for every $x\in[0,1]$?
By now, it is known that each of the de~Rham functions $R_p$, where $p\in(0,1)$, is a solution of the MW--problem, and for any Borel probability measure $\mu$ concentrated on $(0,1)$ the formula $\phi_\mu(x)=\int_{(0,1)}R_p(x)\, d\mu(p)$ defines a solution $\phi_\mu\colon[0,1]\to[0,1]$ of this problem as well. In this paper, we give a new family of solutions of the MW--problem consisting of Cantor-type functions. We also prove that there are strictly increasing solutions of the MW--problem that are not of the above integral form with any Borel probability measure $\mu$.
\end{abstract}

\maketitle


\renewcommand{\theequation}{1.\arabic{equation}}\setcounter{equation}{0}
\section{Introduction}\label{S:1}

During the 47th International Symposium on Functional Equations in 2009, Jacek Wesołowski asked whether the identity on $[0,1]$ is the only increasing and continuous solution $\varphi\colon [0,1]\to [0,1]$ of
\begin{equation}\label{eq:e}
\varphi(x)=\varphi\left(\frac{x}{2}\right)+\varphi\left(\frac{x+1}{2}\right)
-\varphi\left(\frac{1}{2}\right)
\end{equation}
satisfying the following boundary conditions
\begin{equation}\label{eq:c}
\varphi(0)=0\hspace{3ex}\hbox{ and }\hspace{3ex}\varphi(1)=1;
\end{equation}
here and throughout this paper \emph{increasing} means weakly increasing. This question has been posed in connection with studying probability measures in the plane that are invariant by \lq\lq winding\rq\rq\ (see \cite{MisiewiczWesolowski12}). It turns out that twenty-four years earlier Janusz Matkowski posed a problem in \cite{Matkowski85}, which is equivalent to Wesołowski's question. 

A negative answer to Wesołowski's question can be found in \cite[Section~5.C]{Kairies97}, where it is noted that each of the \ignoreSpellCheck{de~Rham} functions satisfies equation \eqref{eq:e};
recall that the \ignoreSpellCheck{de~Rham} function $\varphi_p\colon[0,1]\to\mathbb R$, where $p\in(0,1)$, is the unique bounded solution of the system of equations
\begin{equation*}
\begin{cases}
\varphi_p\left(\frac{x}{2}\right)=p\varphi_p(x) &\text{for }x\in[0,1],\\ \varphi_p\left(\frac{x+1}{2}\right)=(1-p)\varphi_p(x)+p &\text{for }x\in[0,1]
\end{cases} 	
\end{equation*}
(see \cite[Section~2]{Rham56}, cf.~\cite[Section~5.C]{Kairies97}, see also \cite[Theorem~3.3]{Hata85} for a more general system of functional equations). 
It is well known that $\varphi_\frac{1}{2}=\id_{[0,1]}$, the identity function on $[0,1]$, and for each $p\neq\frac{1}{2}$ the \ignoreSpellCheck{de~Rham} function $\varphi_p$ is strictly increasing, singular, and continuous (see p.~106 and p.~102 in~\cite{Rham56}); it is even Hölder continuous (see \cite[Theorem~2.1]{BergKruppel00b}). Moreover, according to \cite[Theorem~1.1]{MorawiecZurcher18} or (stated slightly differently) \cite[Section~5.C]{Kairies97}, for any $p\in(0,1)$ we have
\begin{equation*}
\varphi_p\left(\sum_{n=1}^\infty\frac{x_n}{2^n}\right)=\sum_{n=1}^{\infty}x_n p^{n-\sum_{i=1}^{n-1}x_i}(1-p)^{\sum_{i=1}^{n-1}x_i}\quad\text{for }
(x_n)_{n\in\mathbb N}\in\{0,1\}^{\mathbb N}.
\end{equation*}
In particular,
\begin{equation}\label{eq:13}
\varphi_p\left(\frac{1}{2^n}\right)=p^n\quad\text{for }p\in(0,1)\text{ and }n\in\mathbb N_0.
\end{equation}

From now on, we say that  a \emph{solution of the Matkowski and Wesołowski problem} (\ignoreSpellCheck{MW}--problem for short) is any increasing and continuous function $\varphi\colon[0,1]\to[0,1]$ satisfying \eqref{eq:e} and \eqref{eq:c}. 

A large family of solutions of the \ignoreSpellCheck{MW}--problem containing functions that are not Hölder continuous was constructed in \cite{MorawiecZurcher21} (see Theorem~2.4 and Section~5). To describe this family (in a slightly different way) let us denote by $\mathcal B$ the $\sigma$-algebra of all Borel subsets of $[0,1]$, by $\mathcal M$ the family of all probability measures defined on $\mathcal B$, and by $\mathcal M_{(0,1)}$ the subfamily of all $\mu\in\mathcal M$ with $\mu(\{0,1\})=0$. Then define $\Theta\colon (0,1)\times [0,1]\to [0,1]$ by 
\begin{equation*}
\Theta(p,x)=\varphi_p(x)
\end{equation*}
and note that $\Phi$ is differentiable with respect to the first variable (\hspace{-1.1pt}see \cite[Proposition~3.1]{Kruppel09}; cf.\ \cite[Theorem~4.6]{HataYamaguti84} for the case $p=\frac{1}{2}$) and Hölder continuous with respect to the second variable (see \cite[Theorem~2.1]{BergKruppel00b}). Finally, making use of Lebesgue's Dominated Convergence Theorem, we conclude that for every $\mu\in\mathcal M_{(0,1)}$ the formula
\begin{equation}\label{eq:int}
\phi_\mu(x)=\int_{(0,1)}\Theta(p,x)\,d\mu(p)
\end{equation}
defines a continuous function $\phi_\mu\colon[0,1]\to\mathbb R$ (see \cite[Satz~5.6]{Elstrodt05}). Since $\Theta$ is strictly increasing with respect to the second variable, so is $\phi_\mu$. It is easy to check that $\phi_\mu$ is a solution of the \ignoreSpellCheck{MW}--problem. Therefore, we have the following fact.

\begin{prop}\label{pp:11}
For every $\mu\in\mathcal M_{(0,1)}$ the function $\phi_\mu\colon [0,1]\to [0,1]$ given by~\eqref{eq:int} is a strictly increasing solution of the \ignoreSpellCheck{MW}--problem.
\end{prop}

In 2021, during the talk on the \ignoreSpellCheck{MW}--problem, given by the first author at the Probabilistic Seminar at the Faculty of Mathematics and Information Sciences of the Warsaw University of Technology, the following question was asked.

\begin{problem}\label{prob:12}
Is every solution of the \ignoreSpellCheck{MW}--problem of the form \eqref{eq:int} with some $\mu\in\mathcal M_{(0,1)}$?
\end{problem}

The goal of this paper is to give a negative answer to \cref{prob:12}.
We first prove that there exist solutions of the \ignoreSpellCheck{MW}--problem that are constant on some intervals, which in view of \cref{pp:11} cannot be written in the integral form \eqref{eq:int} with some $\mu\in\mathcal M_{(0,1)}$. Then, we also prove that there are strictly increasing solutions not of that integral form. Moreover, before formulating the two announced results, we make a comment about a connection between solutions of the \ignoreSpellCheck{MW}--problem that are of the integral form \eqref{eq:int} with some $\mu\in\mathcal M_{(0,1)}$ and the moment problem.


\renewcommand{\theequation}{2.\arabic{equation}}\setcounter{equation}{0}
\section{Solutions of the integral form \eqref{eq:int}}\label{S:2}
Note that for each $\mu\in\mathcal M_{(0,1)}$, we have by \eqref{eq:13}
\begin{equation*}
\left(\phi_\mu\left(\frac{1}{2^n}\right)\right)_{n\in\mathbb N_0}
=\left(\int_{(0,1)}p^n\,d\mu(p)\right)_{n\in\mathbb N_0}.
\end{equation*}
We first turn our attention towards a characterization of all sequences that can appear on the right-hand side of the above equality.

If $\mu\in\mathcal M$, then the sequence $(c_j)_{j\in\mathbb N_0}$ given by
\begin{equation*}
c_j=\int_{[0,1]}x^j\,d\mu(x)
\end{equation*}
is said to be the \emph{moment sequence} of $\mu$. 

The following problem is called the \emph{Hausdorff moment problem}: Given a sequence $(c_j)_{j\in\mathbb N_0}$ of real numbers, we ask: when does there exist a $\mu\in\mathcal M$ such that $(c_j)_{j\in\mathbb N_0}$ is the moment sequence of $\mu$?
This particular moment problem is one of a large class of general moment problems (see e.g.\ \cite{Schmudgen17}, \cite{Feller66}). Hausdorff's name is added to this particular moment problem because Hausdorff solved it completely (see \cite{Hausdorff23}, cf.\ \cite{Hausdorff21a,Hausdorff21b}). To formulate Hausdorff's result, we need one more definition.
A sequence $(c_j)_{j\in\mathbb N_0}$ of real numbers is said to be \emph{completely monotone}, if 
\begin{equation*}
\sum_{j=0}^n(-1)^j\binom{n}{j}c_{k+j}\geq 0\quad\hbox{for }k,n\in\mathbb N_0.
\end{equation*}

\begin{theorem}[see {\cite[VII.3, Theorem 1]{Feller66} or \cite[Theorem 3.15]{Schmudgen17}}]\label{th:21}\ 
\begin{enumerate}[label=$(\roman*)$]
\item\label{th:211} Every moment sequence $(c_j)_{j\in\mathbb N_0}$ of $\mu\in\mathcal M$ is completely monotone with $c_0=1$.
\item Every completely monotone sequence $(c_j)_{j\in\mathbb N_0}$ with $c_0=1$ coincides with the 
moment sequence of a unique $\mu\in\mathcal M$.
\end{enumerate}
\end{theorem}

The following two observations show a direct connection between solutions of the \ignoreSpellCheck{MW}--problem that are of the form \eqref{eq:int} with a $\mu\in\mathcal M_{(0,1)}$ and the Hausdorff moment problem.

\begin{prop}\label{pp:22}
Assume that $\mu\in\mathcal M_{(0,1)}$ and let $\phi_\mu\colon [0,1]\to [0,1]$ be given by \eqref{eq:int}. Then the sequence $(\phi_\mu(\frac{1}{2^j}))_{j\in\mathbb N_0}$ is completely monotone and
\begin{equation*}
\lim_{n\to\infty}\sum_{j=0}^{n}(-1)^j\binom{n}{j}\phi_\mu\left(\frac{1}{2^{j}}\right)=0.
\end{equation*}
\end{prop}

\begin{proof}
Fix $k,n\in\mathbb N_0$. By \eqref{eq:13}, we have
\begin{equation*}
\sum_{j=0}^{n}(-1)^j\binom{n}{j}\varphi_p\left(\frac{1}{2^{k+j}}\right)
=p^k\sum_{j=0}^{n}(-1)^j\binom{n}{j}p^j=p^k(1-p)^n
\end{equation*}
for every $p\in(0,1)$. Then 
\begin{align*}
\sum_{j=0}^{n}(-1)^j\binom{n}{j}\phi_\mu\left(\frac{1}{2^{j+k}}\right)
&=\int_{(0,1)}\sum_{j=0}^{n}(-1)^j\binom{n}{j}\Phi\left(p,\frac{1}{2^{j+k}}\right)\,d\mu(p)\\
&=\int_{(0,1)}p^k(1-p)^n\,d\mu(p)\geq 0,
\end{align*}
and according to Lebesgue's Dominated Convergence Theorem we get
\begin{align*}
\lim_{n\to\infty}\sum_{j=0}^{n}(-1)^j\binom{n}{j}\phi_\mu\left(\frac{1}{2^{j}}\right)
=\lim_{n\to\infty}\int_{(0,1)}(1-p)^n\,d\mu(p)=0,
\end{align*}
which completes the proof.
\end{proof}

\begin{prop}\label{pp:23}
Assume that $(c_j)_{j\in\mathbb N_0}$ is a completely monotone sequence with $c_0=1$ that decreases to $0$. If 
\begin{equation*}
\lim_{n\to\infty}\sum_{j=0}^{n}(-1)^j\binom{n}{j}c_{j}=0,
\end{equation*}
then there exists a unique $\mu\in\mathcal M_{(0,1)}$ such that $\phi_\mu\colon[0,1]\to[0,1]$ given by~\eqref{eq:int} is a solution of the \ignoreSpellCheck{MW}--problem with
\begin{equation*}
\phi_\mu\left(\frac{1}{2^j}\right)=c_j\quad\text{for }j\in\mathbb N_0.
\end{equation*}
\end{prop}

\begin{proof}
By assertion (ii) of \cref{th:21} there exists a unique $\mu\in\mathcal M$ such that  $(c_j)_{j\in\mathbb N_0}$ is the moment sequence of $\mu$.
	
We first note that 
\begin{equation*}
0=\lim_{j\to\infty}c_j=\lim_{j\to\infty}\int_{[0,1]}x^j\,d\mu(x)=\mu(\{1\}).
\end{equation*}
This jointly with Lebesgue's Dominated Convergence Theorem yields
\begin{align*}
1&=c_0=\lim_{n\to\infty}\sum_{j=1}^{n}(-1)^{j+1}\binom{n}{j}c_j
=\lim_{n\to\infty}\int_{[0,1]}\sum_{j=1}^{n}(-1)^{j+1}\binom{n}{j}p^j\,d\mu(p)\\
&=\lim_{n\to\infty}\int_{(0,1)}\sum_{j=1}^{n}(-1)^{j+1}\binom{n}{j}p^j\,d\mu(p)
=\lim_{n\to\infty}\int_{(0,1)}\big(1-(1-p)^n\big)\,d\mu(p)\\
&=\int_{(0,1)}1\,d\mu(p)=\mu((0,1)).
\end{align*}
Therefore, $\mu\in\mathcal M_{(0,1)}$.

From \cref{pp:11} we see that the formula \eqref{eq:int} defines a solution of the \ignoreSpellCheck{MW}--problem. By \eqref{eq:13}, we have
\begin{equation*}
c_j=\int_{(0,1)}p^j\,d\mu(p)=\int_{(0,1)}\varphi_p\left(\frac{1}{2^j}\right)\,d\mu(p)
=\phi_\mu\left(\frac{1}{2^j}\right)
\end{equation*}
for every $j\in\mathbb N_0$.
\end{proof}

Let $\varphi\colon[0,1]\to[0,1]$ be a solution of the \ignoreSpellCheck{MW}--problem. 

Assume first that there exists $\mu\in\mathcal M_{(0,1)}$ such that $\varphi=\phi_\mu$. \Cref{pp:22} implies that the sequence $(\varphi(\frac{1}{2^j}))_{j\in\mathbb N_0}$ is completely monotone and 	
\begin{equation}\label{eq:21}
\lim_{n\to\infty}\sum_{j=0}^{n}(-1)^j\binom{n}{j}\varphi\left(\frac{1}{2^{j}}\right)=0.
\end{equation}

Assume now that the sequence $(\varphi(\frac{1}{2^j}))_{j\in\mathbb N_0}$ is completely monotone and \eqref{eq:21} holds. \cref{pp:23} yields the existence of $\mu\in\mathcal M_{(0,1)}$ such that $\phi_\mu\colon[0,1]\to[0,1]$ given by \eqref{eq:int} is a solution of the \ignoreSpellCheck{MW}--problem and
\begin{equation*}
\phi_\mu\left(\frac{1}{2^j}\right)=\varphi\left(\frac{1}{2^j}\right)\quad\text{for }j\in\mathbb N_0.
\end{equation*}
Since both the functions $\varphi$ and $\phi_\mu$ are solutions of the \ignoreSpellCheck{MW}--problem, we also have $\phi_\mu(0)=\varphi(0)$. The question is: Does $\varphi=\phi_\mu$ hold? 
This leads to the following question about a possible characterization of solutions  of the MW--problem that are of integral form \eqref{eq:int}.
\begin{problem}\label{prob:24}
Is a solution $\varphi\colon[0,1]\to[0,1]$ of the MW--problem of the form~\eqref{eq:int} with some $\mathcal M_{(0,1)}$ if and only if the sequence $(\varphi(\frac{1}{2^j}))_{j\in\mathbb N_0}$ is completely monotone and \eqref{eq:21} hold?
\end{problem}
Note that if $\varphi\colon[0,1]\to[0,1]$ is a solution of the MW--problem, then by its monotonicity we have $\varphi(\frac{1}{2^k})\geq\varphi(0)=0$ as well as $\varphi(\frac{1}{2^{k}})
-\varphi(\frac{1}{2^{k+1}})\geq 0$, and finally
\begin{align*}
\varphi\left(\frac{1}{2^{k}}\right)
-2\varphi\left(\frac{1}{2^{k+1}}\right)+\varphi\left(\frac{1}{2^{k+2}}\right)
=\varphi\left(\frac{1}{2^{k+1}}+\frac{1}{2}\right)
	-\varphi\left(\frac{1}{2^{k+2}}+\frac{1}{2}\right)\geq 0
\end{align*}
for every $k\in\mathbb N_0$.


\renewcommand{\theequation}{3.\arabic{equation}}\setcounter{equation}{0}
\section{Answer to \texorpdfstring{\Cref{prob:12}}{Problem~\ref{prob:12}}}\label{S:3}
In this section we give a negative answer to \cref{prob:12} formulating two results. The first one concerns solutions of the \ignoreSpellCheck{MW}--problem in the class of Cantor-like functions, whereas the second result is devoted to solutions of the \ignoreSpellCheck{MW}--problem that are strictly increasing but not of the form \eqref{eq:int}.

Given $m\in\mathbb N$ we put
\begin{equation*}
\mathcal P_m=\left\{(p_0,\ldots, p_{2^m-1})\,\Big|\, p_0,\ldots,p_{2^m-1}\in[0,1)\text{ with } \sum_{k=0}^{2^m-1}p_k=1\right\}
\end{equation*}
and $K=\{0,1,\ldots,2^m-1\}$. For each $k\in K$ define $f_k\colon[0,1]\to[0,1]$ by $f_k(x)=\frac{x+k}{2^m}$ and set $\mathcal F=\big\{f_k\,|\,k\in K\big\}$. The pair $(\mathcal F,P)$ is an \emph{iterated function system with probabilities} (\ignoreSpellCheck{IFSwP} for short). According to \cite[Section~4.4]{Hutchinson81} (cf.\ \cite[Theorem~2.8]{Falconer97}) there exists the unique $\mu_P\in\mathcal M$ satisfying
\begin{equation}\label{eq:inv}
\mu_P(B)=\sum_{k\in K}p_k\mu_P(f_k^{-1}(B))\quad\text{for }B\in\mathcal B([0,1]),
\end{equation}
where $\mathcal B([0,1])$ denotes the family of all Borel subsets of the interval $[0,1]$. Denote by $\Phi_P\colon[0,1]\to[0,1]$ the probability distribution function (pd.\ function for short) of
the unique $\mu_P\in\mathcal M$ satisfying~\eqref{eq:inv}, i.e. $\Phi_P$ and $\mu_P$ are related by the formula
\begin{equation}\label{eq:varP}
\Phi_P(x)=\mu_P([0,x])\quad\text{for }x\in[0,1].
\end{equation}
From now on, for any $P\in\mathcal P_m$ the symbols $\mu_P$ and $\Phi_P$ are fixed for $\mu_P\in\mathcal M$ satisfying \eqref{eq:inv} and its pd.\ function, respectively. 

Note that for $m=1$ we have $\mathcal P_1=\{(p,1-p)\,|\,p\in(0,1)\}$, and according to \cite{MorawiecZurcher18} we conclude that for every $p\in(0,1)$ the \ignoreSpellCheck{de~Rham} function $\varphi_p$ coincides with the function $\Phi_{(p,1-p)}$.
We put
\begin{equation}\label{MW1}
MW_1=\{\varphi_p\,|\,p\in(0,1)\}
\end{equation}
(later in~\eqref{MWp}, we will look at~$WM_m$ for all natural numbers $m>1$).
The family $MW_1$ consists of strictly increasing solutions of the \ignoreSpellCheck{MW}--problem, and it is the base for producing new solutions of the \ignoreSpellCheck{MW}--problem, as described in \cref{S:1}.

Given $P=(p_0,\ldots,p_{2^m-1})\in\mathcal P_m$ we put $K_P=\{k\in K\,|\,p_k\neq 0\}$ and note that $K_P$ contains at least two elements, then \eqref{eq:inv} is equivalent to
\begin{equation}\label{eq:invK}
\mu_P(B)=\sum_{k\in K_P}p_k\mu_P(f_k^{-1}(B))\quad\text{for }B\in\mathcal B([0,1]),
\end{equation}
and $K_P=K$ for every $P\in\mathcal P_1$. 
Put
\begin{equation*}
A_0=[0,1]\quad\text{and}\quad A_n=\bigcup_{k\in K_P}f_k(A_{n-1})\quad\text{for }n\in\mathbb N,
\end{equation*}
and define the \emph{attractor} of the considered \ignoreSpellCheck{IFSwP} by
\begin{equation*}
A_*=\bigcap_{n\in\mathbb N}A_n
\end{equation*}
(see \cite[Section~3.7, Definition~2]{Barnsley88}). It is clear that $A_*$ is a compact set and $A_*=\bigcup_{k\in K_P}f_k(A_*)$. By \cite[Chapter~9.6, Theorem~2]{Barnsley88}, $A_*$ is the support of~$\mu_P$. 
According to \cite[Section~3]{MorawiecZurcher19}, $A^*$ is an uncountable perfect subset of~$\mathbb R$, and if moreover $K_P$ is a proper subset of the set $K$, then $A_*$ is a \emph{Cantor-like set}, i.e.\ uncountable,
perfect, nowhere dense (see \cite[Part~II, Section 1]{WiseHell93}) and of Lebesgue measure zero. Furthermore, from the above construction we have
\begin{equation*}
A_*=\bigcap_{n\in\mathbb N}\left(\bigcup_{k_1,\ldots,k_n\in K_P}
\big[(f_{k_1}\circ\dots\circ f_{k_n})(0),(f_{k_1}\circ\dots\circ f_{k_n})(1)\big]\right),
\end{equation*}
and since an easy induction yields $|(f_{k_1}\circ\dots\circ f_{k_n})(1)-(f_{k_1}\circ\dots\circ f_{k_n})(0)|\leq\frac{1}{2^{nm}}$ for all $k_1,\ldots,k_n\in K_P$ and $n\in\mathbb N$, it follows that $x\in A_*$ if and only if there exists a sequence $(k_n)_{n\in\mathbb N}\in K_P^{\mathbb N}$, called the \emph{address} of $x$ (see~\cite[Section~4.2, Theorem~1 and Definition~2]{Barnsley88}), such that 
\begin{equation*}
x=\lim_{n\to\infty}f_{k_1,\ldots,k_n}(0)=\lim_{n\to\infty}f_{k_1,\ldots,k_n}(1).
\end{equation*}


\subsection{The class of Cantor-like solutions of the \ignoreSpellCheck{MW}--problem}
Throughout this subsection we fix a natural number $m\geq 2$ and $P\in\mathcal P_m$.

The first result is a consequence of \cite[Lemma~4.3]{MorawiecZurcher19}.
We do not repeat the proof but only point out that the statement is a consequence of the fact that $A_*$ is the support of~$\mu_P$ (see \cite[Chapter~9.6, Theorem~2]{Barnsley88}) and the points in~$A_*$ have addresses.

\begin{lemma}\label{lm:31}
For any $x\in [0,1]$ we have $\mu_P(\{x\})=0$. In particular, $\mu_P\in\mathcal M_{(0,1)}$ and its pd.\ function $\Phi_P$ is continuous. 
\end{lemma}

Before formulating the main result of this subsection, we need the following fact.

\begin{prop}\label{pp:32}
The pd.\ function $\Phi_P$ is constant on each component of the set $[0,1]\setminus A_*$, and
\begin{equation}\label{eq:E}
\Phi_P(x)=\sum_{k\in K}\left[\Phi_P\left(\frac{x+k}{2^m}\right)
-\Phi_P\left(\frac{k}{2^m}\right)\right]\quad\text{for }x\in[0,1].
\end{equation}
\end{prop}

\begin{proof}
Using the properties of~$A_*$ that we have already listed (cf.\ \cite[Theorem~4.6]{MorawiecZurcher19}), we deduce that $\Phi_P$ is constant on each component of the set $[0,1]\setminus A_*$. 
To show that~\eqref{eq:E} holds, we first observe that arguing analogously to the proof of Theorem~3.3 in \cite{MorawiecZurcher18}, we have
\begin{equation}\label{eq:EK}
\Phi_P(x)=\sum_{k\in K_P}\left[\Phi_P\left(\frac{x+k}{2^m}\right)
-\Phi_P\left(\frac{k}{2^m}\right)\right]\quad\text{for }x\in[0,1].
\end{equation}
If $k\in K\setminus K_P$, then \eqref{eq:invK} and \cref{lm:31} yield
\begin{align*}
\Phi_P\left(\frac{k+1}{2^m}\right)&=\mu_P\left(\left[0,\frac{k}{2^m}\right]\right)
+\mu_P\left(\left(\frac{k}{2^m},\frac{k+1}{2^m}\right)\right)
+\mu_P\left(\left\{\frac{k+1}{2^m}\right\}\right)\\
&=\Phi_P\left(\frac{k}{2^m}\right)+\sum_{l\in K_P}p_{l}\mu_P(\emptyset)
=\Phi_P\left(\frac{k}{2^m}\right),
\end{align*}
which together with \eqref{eq:EK} gives \eqref{eq:E}.
\end{proof}

The main result of this subsection together with \cref{pp:11} gives a negative answer to \Cref{prob:12} and reads as follows.

\begin{theorem}\label{th:33}
The function $\varphi_P\colon[0,1]\to[0,1]$ defined by
\begin{equation}\label{eq:MW}
\varphi_P(x)=\frac{1}{m}\sum_{i=0}^{m-1}\sum_{k=0}^{2^i-1}
\left[\Phi_P\left(\frac{x+k}{2^i}\right)-\Phi_P\left(\frac{k}{2^i}\right)\right]
\end{equation}
is a solution of the \ignoreSpellCheck{MW}--problem. 

Moreover, if $K_P\neq K$, then $\varphi_P$ is not strictly increasing.
\end{theorem}

\begin{proof}
Obviously, $\varphi_P(0)=0$. To see that $\varphi_P(1)=1$ it suffices to note that $\mu_P\in\mathcal M$ implies $\Phi_P(1)=1$. Since $\Phi_P$ is increasing, so is $\varphi_P$. The continuity of $\varphi_P$ follows from \cref{lm:31}. To see that $\varphi_P$ satisfies \eqref{eq:e} we fix $x\in[0,1]$ and observe that applying \eqref{eq:E}, which holds by \cref{pp:32}, we have
\begin{align*}
\varphi_P(x)&=\frac{1}{m}\left(\Phi_P(x)+\sum_{i=1}^{m-1}\sum_{k=0}^{2^i-1}\left[\Phi_P\left(\frac{x+k}{2^i}\right)-\Phi_P\left(\frac{k}{2^i}\right)\right]\right)\\
&=\frac{1}{m}\sum_{i=1}^{m}\sum_{k=0}^{2^i-1}\left[\Phi_P\left(\frac{x+k}{2^i}\right)-\Phi_P\left(\frac{k}{2^i}\right)\right]\\
&=\frac{1}{m}\sum_{i=0}^{m-1}\sum_{k=0}^{2^{i+1}-1}\left[\Phi_P\left(\frac{x+k}{2^{i+1}}\right)-\Phi_P\left(\frac{k}{2^{i+1}}\right)\right]\\
&=\frac{1}{m}\sum_{i=0}^{m-1}\sum_{k=0}^{2^{i}-1}\left[\Phi_P\left(\frac{x+2k}{2^{i+1}}\right)-\Phi_P\left(\frac{2k}{2^{i+1}}\right)\right]\\
&\quad+\frac{1}{m}\sum_{i=0}^{m-1}\sum_{k=0}^{2^{i}-1}\left[\Phi_P\left(\frac{x+2k+1}{2^{i+1}}\right)-\Phi_P\left(\frac{2k+1}{2^{i+1}}\right)\right]\\
&=\frac{1}{m}\sum_{i=0}^{m-1}\sum_{k=0}^{2^{i}-1}\left[\Phi_P\left(\frac{\frac{x}{2}+k}{2^{i}}\right)-\Phi_P\left(\frac{k}{2^{i}}\right)\right]\\
&\quad+\frac{1}{m}\sum_{i=0}^{m-1}\sum_{k=0}^{2^{i}-1}\left[\Phi_P\left(\frac{\frac{x+1}{2}+ k}{2^{i}}\right)-\Phi_P\left(\frac{2k+1}{2^{i+1}}\right)\right]\\
&=\varphi_P\left(\frac{x}{2}\right)+\frac{1}{m}\sum_{i=0}^{m-1}\sum_{k=0}^{2^{i}-1}\left[\Phi_P\left(\frac{\frac{x+1}{2}+ k}{2^{i}}\right)-\Phi_P\left(\frac{k}{2^{i}}\right)\right]\\
&\quad+\frac{1}{m}\sum_{i=0}^{m-1}\sum_{k=0}^{2^i-1}\left[\Phi_P\left(\frac{k}{2^i}\right)-\Phi_P\left(\frac{1+2k}{2^{i+1}}\right)\right]\\
&=\varphi_P\left(\frac{x}{2}\right)+\varphi_P\left(\frac{x+1}{2}\right)-\varphi_P\left(\frac{1}{2}\right).
\end{align*}

Finally, observe that if $K_P\neq K$, then $A_*$ is a Cantor-like set, and hence the set
\begin{equation*}
A=\bigcup_{i=0}^{m-1}\bigcup_{k=0}^{2^i-1}2^i\left(A_*-\frac{k}{2^i}\right)
\end{equation*}
is closed and of Lebesgue measure zero. Therefore, by \cref{pp:32} we conclude that $\varphi_P$ is constant on each component of the set $[0,1]\setminus A$, which means that it is not strictly increasing.
\end{proof} 

For each $P\in\mathcal P_m$, we denote by $\varphi_P$ the solution of the \ignoreSpellCheck{MW}--problem given by \eqref{eq:MW} and put 
\begin{equation}\label{MWp}
MW_m=\{\varphi_P\,|\,P\in\mathcal P_m\}.
\end{equation}
Let us recall that the family $MW_1$ was defined in~\eqref{MW1} and each of its member can be written in the form \eqref{eq:int}. However, \cref{th:33} says that the family $MW_m$ (recall that $m\geq 2$) contains solutions of the \ignoreSpellCheck{MW}--problem that are not strictly increasing and, in view of \cref{pp:11}, cannot have the form \eqref{eq:int} with any $\mu\in\mathcal M_{(0,1)}$.


\subsection{The class of strictly increasing solutions of the \ignoreSpellCheck{MW}--problem that are not of the form \eqref{eq:int}}
The aim of this subsection is to show that for each integer number $m\geq 2$ the family $MW_m$ contains solutions of the \ignoreSpellCheck{MW}--problem that are strictly increasing but not of the form \eqref{eq:int} with any $\mu\in\mathcal M_{(0,1)}$.

The first assertion of the next lemma can be found as Lemma~3.1 in~\cite{MorawiecZurcher21}.
The second assertion is a consequence of the first one and the fact that if $\varphi\colon[0,1]\to[0,1]$ is a solution of the \ignoreSpellCheck{MW}--problem, then so is the function $\psi\colon[0,1]\to[0,1]$ defined by $\psi(x)=1-\varphi(1-x)$.

\begin{lemma}\label{lm:34}
Let $\varphi$ be a solution of the \ignoreSpellCheck{MW}--problem.
\begin{enumerate}[label=(\roman*)]	
\item\label{D0} If $\liminf_{x\to 0+} \frac{\varphi(x)}{x}=0$, then $\varphi$ is singular.
\item\label{D1} If $\liminf_{x\to 1-} \frac{\varphi(1)-\varphi(x)}{1-x}=0$, then $\varphi$ is singular.
\end{enumerate}
\end{lemma}

\begin{theorem}\label{th:35}
For each integer number $m\geq 2$ there exists $\varphi_P\in MW_m$ that is strictly increasing but not of the form \eqref{eq:int} with any $\mu\in\mathcal M_{(0,1)}$.
\end{theorem}

\begin{proof}
Fix an integer number $m\geq 2$. Fix also $P\in\mathcal P_m$ such that $\varphi_P$ is not strictly increasing; this is possible in view of \cref{th:33}. According to \cref{pp:11} we deduce that $\varphi_P$ is not of the form \eqref{eq:int} with any $\mu\in\mathcal M_{(0,1)}$, and since any solution of the \ignoreSpellCheck{MW}--problem is a convex combination of the identity function on $[0,1]$ and a singular solution of the \ignoreSpellCheck{MW}--problem (see \cite[Theorem~1.1~(i) and Remark~2.2]{MorawiecZurcher18}), we conclude that $\varphi_P$ is singular. 
Fix also $\alpha\in(0,1)$ and define $\psi\colon[0,1]\to[0,1]$ putting
\begin{equation*}
\psi(x)=\alpha x+(1-\alpha)\varphi_P(x).
\end{equation*}
Clearly, $\psi$ is a solution of the \ignoreSpellCheck{MW}--problem. We will show that there is no $\mu\in\mathcal M_{(0,1)}$ such that $\psi=\phi_\mu$, where $\phi_\mu$ is given by \eqref{eq:int}. 

Assume by contradiction that there exists $\mu\in\mathcal M_{(0,1)}$ such that $\psi=\phi_\mu$. 

Define $f,g\colon[0,1]\to[0,1]$ by
\begin{equation*}
f(x)=\int_{(0,\frac{1}{2})}\Phi(p,x)\,d\mu(p)\quad\text{and}\quad
g(x)=\int_{(\frac{1}{2},1)}\Phi(p,x)\,d\mu(p).
\end{equation*}
Then for every $x\in[0,1]$ we have
\begin{equation}\label{eq:37}
\alpha x+(1-\alpha)\varphi_P(x)=f(x)+\mu\left(\left\{\frac{1}{2}\right\}\right)x+g(x).
\end{equation}
Assume for a moment that $f$ and $g$ are singular functions. 
Then, differentiating equality \eqref{eq:37}, we get $\alpha=\mu\left(\left\{\frac{1}{2}\right\}\right)$, which implies that
\begin{equation*}
\varphi_P(x)=\frac{f(x)+g(x)}{1-\alpha}=\frac{1}{1-\alpha}\int_{(0,1)\setminus \{\frac{1}{2}\}}\Phi(p,x)\,d\mu(p).
\end{equation*}
Since the formula 
\begin{equation*}
\nu(B)=\frac{1}{1-\alpha}\mu\left(B\setminus\left\{\frac{1}{2}\right\}\right)\quad\text{for every }B\in\mathcal B, 
\end{equation*}
defines a measure belonging to $\mathcal M_{(0,1)}$, we obtain
\begin{equation*}
\varphi_P(x)=\int_{(0,1)}\Phi(p,x)\,d\nu(p),
\end{equation*}
a contradiction.
	
To complete the proof, it suffices to show that $f$ and $g$ are singular functions.

If $\mu((0,\frac{1}{2}))=0$, then $f=0$ and it is singular. Let $\mu((0,\frac{1}{2}))>0$. Then the function $F=\frac{1}{\mu((0,\frac{1}{2}))}f$ is a solution of the \ignoreSpellCheck{MW}--problem as it is of the form~\eqref{eq:int}. Applying~\eqref{eq:13} and the Lebesgue's Dominated Convergence Theorem, we get
\begin{align*}
\lim_{n\to \infty}\frac{F\left(\frac{1}{2^n}\right)-F(0)}{\frac{1}{2^n}}
&=\frac{1}{\mu((0,\frac{1}{2}))}\lim_{n\to\infty}2^n\int_{(0,\frac{1}{2})}\Phi\left(p,\frac{1}{2^n}\right)\,d\mu(p)\\
&=\frac{1}{\mu((0,\frac{1}{2}))}\lim_{n \to \infty}2^n\int_{(0,\frac{1}{2})}p^nd\mu(p)\\
&=\frac{1}{\mu((0,\frac{1}{2}))}\lim_{n \to \infty}\int_{(0,\frac{1}{2})}(2p)^nd\mu(p)=0.
\end{align*}
Consequently, assertion \cref{D0} of \cref{lm:34} yields the singularity of $F$ and, therefore, $f$.

Similarly, if $\mu((\frac{1}{2},1))=0$, then $g=0$ and it is singular. Let $\mu((\frac{1}{2},1))>0$. Then the function $G=\frac{1}{\mu((\frac{1}{2},1))}g$ is a solution of the \ignoreSpellCheck{MW}--problem. Applying~\eqref{eq:13}, the fact that $\Phi(p,x)=1-\Phi(1-p,1-x)$ for all $p\in(0,1)$ and $x\in[0,1]$ (see \cite[Proposition 2.3]{BergKruppel00a}), and the Lebesgue's Dominated Convergence Theorem, we get
\begin{equation*}
\begin{split}
\lim_{n\to\infty}\frac{G(1)-G\left(1-\frac{1}{2^n}\right)}{\frac{1}{2^n}}
&=\frac{1}{\mu((\frac{1}{2},1))}\lim_{n\to\infty}
\int_{(\frac{1}{2},1)}2^n\!\!\left[1-\Phi\!\left(p,1-\frac{1}{2^n}\right)\right]d\mu(p)\\
&=\lim_{n\to\infty}\frac{1}{\mu((\frac{1}{2},1))}
\int_{(\frac{1}{2},1)}2^n\Phi\!\left(1-p,\frac{1}{2^n}\right)\,d\mu(p)\\
&=\frac{1}{\mu((\frac{1}{2},1))}\lim_{n \to \infty}\int_{(\frac{1}{2},1)}(2(1-p))^nd\mu(p)=0.
\end{split}
\end{equation*}
Finally, by \cref{D1} of \cref{lm:34}, we obtain the singularity of $G$ and thus $g$.
\end{proof}


\renewcommand{\theequation}{4.\arabic{equation}}\setcounter{equation}{0}
\section{More new solutions of the \ignoreSpellCheck{MW}--problem}\label{S:4}
In this section, we will provide a way to define new solutions of the \ignoreSpellCheck{MW}--problem. However, we start with two observations.

\begin{prop}\label{pp:41}
For every $m\in\mathbb N$ we have $MW_m\subset MW_{2m}$.
\end{prop}

\begin{proof}
Fix $P=(p_0,\ldots,p_{2^m-1})\in\mathcal P_m$ and consider the \ignoreSpellCheck{IFSwP} $(\widetilde{\mathcal F},\widetilde{P})$, where $\widetilde{\mathcal F}=\{f_k\circ f_l\,|\,k,l\in K\}$ and $\widetilde{P}=(p_kp_l)_{k,l\in K}$. Using \eqref{eq:inv}, for every $B\in\mathcal B([0,1])$, we get
\begin{equation*}
\mu_P(B)=\sum_{k\in K}p_k\sum_{l\in K}p_l\mu_P(f_l^{-1}(f_k^{-1}(B)))=\sum_{k,l\in K}p_kp_l\mu_P((f_k\circ f_l)^{-1}(B)).
\end{equation*}
By the uniqueness of the invariant measure $\mu_{\widetilde{P}}$ for $(\widetilde{\mathcal F},\widetilde{P})$, we have $\mu_P=\mu_{\widetilde{P}}$, and hence $\Phi_P=\Phi_{\widetilde{P}}$. Then, making use of \eqref{eq:E}, for every $x\in[0,1]$, we obtain
\begin{align*}
\varphi_P(x)&=\frac{1}{m}\sum_{i=0}^{m-1}\sum_{k=0}^{2^i-1}
\left[\Phi_P\left(\frac{x+k}{2^i}\right)-\Phi_P\left(\frac{k}{2^i}\right)\right]\\
&=\frac{1}{m}\sum_{i=0}^{m-1}\sum_{k=0}^{2^i-1}
\sum_{l=0}^{2^m-1}\left[\Phi_P\left(\frac{x+k+2^il}{2^{m+i}}\right)-\Phi_P\left(\frac{k+2^i l}{2^{m+i}}\right)\right]\\
&=\frac{1}{m}\sum_{i=0}^{m-1}\sum_{k=0}^{2^{m+i}-1}
\left[\Phi_P\left(\frac{x+k}{2^{m+i}}\right)-\Phi_P\left(\frac{k}{2^{m+i}}\right)\right]\\
&=\frac{1}{m}\sum_{i=m}^{2m-1}\sum_{k=0}^{2^i-1}
\left[\Phi_P\left(\frac{x+k}{2^i}\right)-\Phi_P\left(\frac{k}{2^i}\right)\right],
\end{align*}
and hence, remembering that $\Phi_P=\Phi_{\widetilde{P}}$, we arrive at
\begin{align*}
\varphi_P(x)&=\frac{1}{2}\varphi_P(x)+\frac{1}{2}\varphi_P(x)
=\frac{1}{2m}\sum_{i=0}^{m-1}\sum_{k=0}^{2^i-1}\left[\Phi_P\left(\frac{x+k}{2^i}\right)-\Phi_P\left(\frac{k}{2^i}\right)\right]\\
&\quad + \frac{1}{2m}\sum_{i=m}^{2m-1}\sum_{k=0}^{2^i-1}
\left[\Phi_P\left(\frac{x+k}{2^i}\right)-\Phi_P\left(\frac{k}{2^i}\right)\right]\\
&=\frac{1}{2m}\sum_{i=0}^{2m-1}\sum_{k=0}^{2^i-1}\left[\Phi_P\left(\frac{x+k}{2^i}\right)-\Phi_P\left(\frac{k}{2^i}\right)\right]\\
&=\frac{1}{2m}\sum_{i=0}^{2m-1}\sum_{k=0}^{2^i-1}\left[\Phi_{\widetilde{P}}\left(\frac{x+k}{2^i}\right)-\Phi_{\widetilde{P}}\left(\frac{k}{2^i}\right)\right]=\varphi_{\widetilde{P}}(x).
\end{align*}
In consequence, $\varphi_P=\varphi_{\widetilde{P}}\in MW_{2m}$.
\end{proof}

To formulate the next lemma, we adopt the convention that $\sum_{k=0}^{-1}p_k=0$.

\begin{lemma}\label{lm:42}
Let $P\in\mathcal P_m$. Then 
\begin{equation*}
\Phi_P\left(\frac{x+l}{2^m}\right)=\sum_{k=0}^{l-1}p_k+p_l\Phi_P(x)
\end{equation*}
for all $l\in K$ and $x\in[0,1]$. 
\end{lemma}

\begin{proof}
Fix $l\in K$ and $x\in[0,1]$. Then
\begin{align*}
\Phi_P\left(\frac{x+l}{2^m}\right)&=\mu_P\left(\left[0,\frac{x+l}{2^m}\right]\right)
=\sum_{k\in K}p_k\mu_P\left(f_k^{-1}\left(\left[0,\frac{x+l}{2^m}\right]\right)\right)\\
&=\sum_{k=0}^{l-1}p_k\mu_P([0,1])+p_l\mu_P([0,x])+\sum_{k=l+1}^{2^m-1}p_k\mu_P(\emptyset)\\
&=\sum_{k=0}^{l-1}p_k+p_l\Phi_P(x),
\end{align*}
and the proof is complete.
\end{proof}

\begin{prop}\label{pp:43}
For all distinct $m,n\in\mathbb N$ we have $MW_m\neq MW_n$.
\end{prop}

\begin{proof}
Fix $m\in \mathbb N$.
To get the required assertion we fix $P=(0,\ldots,0,p,1-p)\in\mathcal P_m$ and note that it is sufficient to prove that 
\begin{equation}\label{eq:41}
\varphi_P(x)=0\quad\text{if and only if}\quad x\in \left[0,\frac{2^m-2^{m-1}-1}{2^m-1}\right].
\end{equation}
Define recursively two sequences $(x_q)_{q\in\mathbb N_0}$ and $(y_q)_{q\in\mathbb N_0}$ putting
\begin{equation*}
x_0=\frac{2^m-2}{2^m},\quad y_0=\frac{2^m-1}{2^m}
\end{equation*} 
and
\begin{equation*}
x_{q}=\frac{x_{q-1}+2^m-2}{2^m},\quad y_{q}=\frac{y_{q-1}+2^m-2}{2^m}
\quad\text{for every }m\in\mathbb N.
\end{equation*} 
It is easy to check that $(x_q)_{q\in\mathbb N_0}$ is strictly increasing, $(y_q)_{q\in\mathbb N_0}$ is strictly decreasing and  $\lim_{q\to\infty}x_q=\lim_{q\to\infty}y_q=\frac{2^m-2}{2^m-1}$.

We begin with proving that 
\begin{equation}\label{eq:42}
\Phi_P(x)=0\quad\text{if and only if}\quad x\in \left[0,\frac{2^m-2}{2^m-1}\right].
\end{equation}
First, we will show that $\Phi_P(x)=0$ for every $x\in [0,\frac{2^m-2}{2^m-1}]$. Since $\Phi_P$ is increasing and continuous, and $(x_q)_{q\in\mathbb N_0}$ is increasing with $\lim_{q\to\infty}x_q=\frac{2^m-2}{2^m-1}$, it suffices to show that $\Phi_P(x_q)=0$ for every $q\in\mathbb N_0$. We will do it by induction. For $q=0$, \cref{lm:42} gives
\begin{equation*}
\Phi_P(x_0)=\Phi_P\left(\frac{2^m-2}{2^m}\right)=p\Phi_P(0)=0.    
\end{equation*}
If $q\in\mathbb N$ and $\Phi_P(x_{q-1})=0$, then applying again \cref{lm:42} we get
\begin{equation*}
\Phi_P(x_{q})=p\Phi_P(x_{q-1})=0.    
\end{equation*}
It remains to show that $\Phi_P(x)>0$ for every $x\in (\frac{2^m-2}{2^m-1},1]$. Again, since $\Phi_P$ is increasing and $(y_q)_{q\in\mathbb N_0}$ is decreasing with $\lim_{q\to\infty}y_q=\frac{2^m-2}{2^m-1}$, it suffices to shown that $\Phi_P(y_q)>0$ for every $q\in\mathbb N_0$. As before we will proceed by induction. For $q=0$, \cref{lm:42} gives
\begin{equation*}
\Phi_P(y_0)=\Phi_P\left(\frac{2^m-1}{2^m}\right)=p+(1-p)\Phi_P(0)=p>0.
\end{equation*}
If $q\in\mathbb N$ and $\Phi_P(y_{q-1})>0$, then applying again \cref{lm:42} we get
\begin{equation*}
\Phi_P(y_{q})=p\Phi_P(y_{q-1})>0.    
\end{equation*}
and the proof of \eqref{eq:42} is complete.

Now, we pass to the proof of \eqref{eq:41}.
 
Fix $x\in[0,\frac{2^m-2^{m-1}-1}{2^m-1}]$. If $i\in \{0,\ldots,m-1\}$ and $k\in \{0,\ldots,2^i-1\}$, then
\begin{equation*}
\begin{split}
0&\leq\frac{x+k}{2^i}\leq\frac{x+2^i-1}{2^i}\leq\frac{x+2^{m-1}-1}{2^{m-1}}
\leq \frac{\frac{2^m-2^{m-1}-1}{2^m-1}+2^{m-1}-1}{2^{m-1}}\\
&=\frac{2^m-2}{2^m-1},
\end{split}
\end{equation*}
and \eqref{eq:42} implies $\varphi_P(x)=0$. 

Fix now $x\in(\frac{2^m-2^{m-1}-1}{2^m-1},\frac{2^m-2^{m-1}}{2^m-1}]$. If
$i\in \{0,\ldots,m-1\}$ and $k\in \{0,\ldots,2^i-1\}$, then $\frac{x+k}{2^i}\leq 1$ and
\begin{equation*}
\frac{2^m-2}{2^m-1}<\frac{x+k}{2^i}\quad\text{if and only if}\quad i=m-1\text{ and }k=2^{m-1}-1.
\end{equation*}
This together with \eqref{eq:42} yields
\begin{equation}
\begin{split}\label{comp pos}
\varphi_P(x)&=\frac{1}{m}\left[\Phi_P\left(\frac{x+2^{m-1}-1}{2^{m-1}}\right)
-\Phi_P\left(\frac{2^{m-1}-1}{2^{m-1}}\right)\right]\\
&=\frac{1}{m}\Phi_P\left(\frac{x+2^{m-1}-1}{2^{m-1}}\right)>0,
\end{split}
\end{equation}
and since $\varphi_P$ is increasing, we conclude that \eqref{eq:41} holds.
\end{proof} 

\begin{remark}\label{rm:44}
Let $m\geq 2$ be an integer number. If $P=(0,\ldots,0,p,1-p)\in\mathcal P_m$, then $\varphi_P\left(\frac{1}{2}\right)=\frac{p}{m}$ and $\varphi_P\left(\frac{1}{2^{j+1}}\right)=0$ for every $j\in\mathbb N$.
\end{remark}
\begin{proof}
Fix $P=(0,\ldots,0,p,1-p)\in\mathcal P_m$. From the proof of \cref{pp:43} we know that \eqref{eq:41} and \eqref{eq:42} hold. Since
\begin{equation*}
\frac{1}{3}\leq\frac{2^m-2^{m-1}-1}{2^m-1}<\frac{1}{2}<\frac{2^m-2^{m-1}}{2^m-1},
\end{equation*}
\eqref{eq:41} yields $\varphi_P\left(\frac{1}{2^{j+1}}\right)=0$ for every $j\in\mathbb N$.
Using~\eqref{comp pos},
\begin{equation*}
\varphi_P\left(\frac{1}{2}\right)=\frac{1}{m}\Phi_P\left(\frac{\frac{1}{2}+2^{m-1}-1}{2^{m-1}}\right)=\frac{1}{m}\Phi_P\left(\frac{2^m-1}{2^m}\right)=\frac{p}{m},
\end{equation*}
ending the proof.
\end{proof}
If $\varphi_P$ is the solution of the \ignoreSpellCheck{MW}--problem from \cref{rm:44}, then 
\begin{equation*}
\lim_{n\to\infty}\sum_{j=0}^{n}(-1)^j\binom{n}{j}\varphi_P\left(\frac{1}{2^{j}}\right)=
\lim_{n\to\infty}\left[\varphi_P(1)-n\varphi_P\left(\frac{1}{2}\right)\right]=-\infty.
\end{equation*}
In particular, $\varphi_P$ does not satisfy the Hausdorff moment problem. Recall (see \cref{pp:11}), that $\varphi_P$ is also not of the integral form \eqref{eq:int}. This sheds some light on \cref{prob:24}. 
\\
Let us also note that \cref{rm:44} implies that for all integer numbers $n>m\geq 2$ there are $\varphi_P\in MW_m$ and $\varphi_{P'}\in MW_n$ such that $\varphi_P(\frac{1}{2^j})=\varphi_{P'}(\frac{1}{2^j})$ for every $j\in\mathbb N$, but $\varphi_P\neq\varphi_{P'}$.

Let us finish this section by showing how to produce more new solutions of the \ignoreSpellCheck{MW}--problem. For this purpose we put $\Delta_m=\{P\in\mathcal P_m\,|\,K_P=K\}$. Clearly, $\Delta_m\subset[0,1]^{2^m}$. Denote by $\mathcal B([0,1]^{2^m})$ the family of all Borel subsets of the cube $[0,1]^{2^m}$ and by $\mathcal M_{\Delta_m}$ the family of all probability measure defined on $\mathcal B([0,1]^{2^m})$ supported on $\Delta_m$, hence $\nu\in\mathcal M_{\Delta_m}$ if $\nu(\Delta_m)=1$. Following the idea of \cite{MorawiecZurcher21}, explained in the paragraph just before \cref{pp:11} in \cref{S:1}, we define $\Psi\colon\Delta_m\times [0,1]\to [0,1]$ by 
\begin{equation*}
\Psi(P,x)=\varphi_P(x)
\end{equation*}
and note that $\Psi$ is increasing and continuous with respect to the second variable. Then for every $\nu\in\mathcal M_{\Delta_m}$ the formula
\begin{equation*}\label{eq:int-new}
\psi_\nu(x)=\int_{\Delta_m}\Psi(P,x)\,d\nu(P)
\end{equation*}
defines a solution of the \ignoreSpellCheck{MW}--problem. Next given a sequences $(\alpha_n)_{n\in\mathbb N}$ of nonnegative real numbers such that $\sum_{n\in\mathbb N}\alpha_n=1$ and a sequence $(\nu_n)_{n\in\mathbb N}$ with $\nu_n\in\mathcal M_{\Delta_n}$ for every $n\in\mathbb N$, we put
\begin{equation}\label{psiSum}
\psi=\sum_{n\in\mathbb N}\alpha_n\psi_{\nu_n}.	
\end{equation}
It is clear that $\psi$ is an increasing function, $\psi(0)=0$, $\psi(1)=1$, and $\psi(x)=\psi(\frac{x}{2})+\psi(\frac{x+1}{2})-\psi(\frac{1}{2})$ for every $x\in[0,1]$.
By the Weierstrass M\nobreakdash-test, $\psi$ is continuous, hence $\psi$ is
a solution of the \ignoreSpellCheck{MW}--problem.
In this manner we can produce a large family of solutions of the \ignoreSpellCheck{MW}--problem. However, we still do not know if we can describe all solutions of the \ignoreSpellCheck{MW}--problem by using the above procedure. This leads to the following question.

\begin{problem}
Can every solution of the \ignoreSpellCheck{MW}--problem be obtained as described above, i.e.\ is of the form~\eqref{psiSum}?
\end{problem}


\section*{Acknowledgement}
The research was supported by the University of Silesia  Mathematics Department (Iterative Functional Equations and Real Analysis program).

\subsection*{Author contributions} All work regarding this manuscript was carried out by
a joint effort of the two authors.

\subsection*{Data availability statement} Data sharing not applicable to this article as no
datasets were generated or analysed during the current study.

\section*{Declarations}
\subsection*{Conflict of interest} The authors declare no competing interests.

\bibliographystyle{plain}
\bibliography{bibliography}
\end{document}